\documentclass[11pt, a4paper]{amsart}

\usepackage{amsmath,amssymb,amsthm,geometry,mathtools}
\usepackage[hidelinks]{hyperref}
\usepackage{tikz}
\usepackage{placeins,needspace}

\newtheorem{theorem}{Theorem}[section]
\newtheorem{lemma}[theorem]{Lemma}
\newtheorem{proposition}[theorem]{Proposition}
\newtheorem{corollary}[theorem]{Corollary}

\theoremstyle{definition}
\newtheorem{definition}[theorem]{Definition}
\newtheorem{example}[theorem]{Example}

\theoremstyle{remark}
\newtheorem{remark}[theorem]{Remark}

\author{Bobo Hua}
\address{Bobo Hua: School of Mathematical Sciences, LMNS,
	Fudan University, Shanghai 200433, China; Shanghai Center for
	Mathematical Sciences, Fudan University, Shanghai 200438,
	China.}
\email{bobohua@fudan.edu.cn}

\author{Yong Lin}
\address{Yau Mathematical Sciences Center; Department of Mathematical Sciences, Tsinghua University, Beijing, 100084, China}
\email{yonglin@tsinghua.edu.cn}

\author{Haohang Zhang}
\address{Yau Mathematical Sciences Center; Department of Mathematical Sciences, Tsinghua University, Beijing, 100084, China}
\email{zhanghh22@tsinghua.org.cn}

\title{Solvability of semilinear elliptic equations on infinite graphs}


\begin{document}
	
\begin{abstract}
	We develop a constructive method for solving semilinear elliptic equations $\Delta u(x)=f(x,u(x))$ on locally finite, connected infinite graphs with layered structure. Using Eidelheit's theorem, we establish coupling criteria ensuring that arbitrary initial-layer data extend to global solutions for every $f$. We apply combinatorial criteria to prove solvability on leafless infinite trees, integer lattices, the triangular and hexagonal lattices, and a Cayley graph of the discrete Heisenberg group. We further establish solvability for a broad class of Cayley graphs of semidirect products $G\cong\mathbb Z\ltimes_\theta H$. In particular, $\Delta u=e^u$ has infinitely many solutions on $\mathbb Z^2$, but none of finite energy. We also extend the method to the bi-Laplacian under two-step coupling conditions, to the $p$-Laplacian under a unique-neighbor condition, and to magnetic Laplacians.
\end{abstract}
	
	\maketitle
	
	\section{Introduction}
	\label{sec:intro}
	
	The starting point of this work was the following question: does the nonlinear exponential equation
	\begin{equation}
		\label{eq:motivating_exponential}
		\Delta u=e^u
		\qquad\text{on }\mathbb Z^2
	\end{equation}
	admit a global solution? With our sign convention,~\eqref{eq:motivating_exponential} corresponds, up to normalization, to the prescribed curvature equation with constant curvature $-1$. The corresponding continuous equation $\Delta_{\mathbb R^2}u=e^u$ does not admit an entire $C^2$-solution. This was proved for the exponential nonlinearity by Wittich~\cite{Wittich1943} and follows more generally from the Keller--Osserman nonexistence theory for $\Delta u\geq f(u)$~\cite{Keller1957,Osserman1957}. 
	
	The Liouville equation, corresponding to the prescribed curvature $+1$, is $\Delta_{\mathbb R^2}u+e^u=0$ under the same convention. Chen and Li~\cite{CL1991} classified all of its finite-mass solutions. On the lattice $\mathbb Z^2$, Ge, Hua, and Jiang~\cite{GHJ2018} obtained a uniform energy lower bound for solutions of the corresponding discrete equation $\Delta u+e^u=0$ and raised the existence of finite-energy solutions as a natural problem. This problem was recently resolved by Chen and Hua~\cite{CH2025}, who constructed finite-energy solutions on $\mathbb Z^2$.
	
	Following the terminology of~\cite{GHJ2018,CH2025}, we call
	\[
		\mathcal E(u):=\sum_{x\in\mathbb Z^2}e^{u(x)}
	\]
	the energy of $u$. Strict subharmonicity alone gives no contradiction on an infinite lattice, where a global maximum need not be attained. We prove the following.
	
	\begin{corollary}
		\label{cor:exponential_lattice}
		On the two-dimensional integer lattice $\mathbb Z^2$, the equation $\Delta u=e^u$ admits infinitely many global solutions, but no solutions of finite energy. That is, there exists no solution satisfying
		\[
		\sum_{x\in\mathbb Z^2} e^{u(x)}<\infty.
		\]
	\end{corollary}
	
	This separates unrestricted solvability from finite-energy existence.
	
	Nonlinear elliptic equations on finite graphs and bounded graph domains are commonly treated by spectral and variational methods~\cite{Chung1997,Grigor2018,GLY2016Yamabe,GLY2016KW,Ge2017}. On infinite locally finite graphs, existence theories usually impose structural hypotheses and seek positive, weak, finite-energy, or otherwise constrained solutions. Representative results cover Yamabe, Kazdan--Warner, $p$-Laplacian, and poly-Laplacian equations~\cite{GLY2017,GeJiangYamabe2018,GJ2018,ChangZhang2021,PinamontiStefani2022,ManZhang2024}, power-type equations on Cayley graphs~\cite{HuaLiWang2023}, and Schr\"odinger or diffusion equations in prescribed function spaces~\cite{SchmidtZimmermann2025,BerchioEtAl2026}.
	
	We instead work in the full product space $\mathbb R^V:=\{u:V\to\mathbb R\}$ and ask whether arbitrary data on one or two initial layers extend to a solution for an arbitrary source function $f:V\times\mathbb R\to\mathbb R$. No positivity or function-space constraint is imposed. Here and below, Yamabe, Kazdan--Warner, and Liouville-type equations refer only to the algebraic form of their nonlinearities, not to positive, finite-energy, or finite-mass solutions.
	
	For the linear background, surjectivity of the combinatorial Laplacian on an infinite locally finite graph was proved by projective-limit and Fr\'echet-space methods~\cite{Dodziuk2012,Kalmes2016}. Finite-hopping analogues include magnetic Schr\"odinger operators on discrete vector bundles~\cite{KobersteinSchmidt2020}; translation-invariant finite-hopping operators on Cayley graphs are also related to linear cellular automata and group-ring methods~\cite{CeccheriniCoornaert2006,CeccheriniCoornaert2008}.
	
	The main result is a global extension theorem for arbitrary nonlinearities. We use layers $V_n$, indexed by $\mathcal I\in\{\mathbb N,\mathbb Z\}$, with edges joining vertices in the same or consecutive layers; see Definition~\ref{def:layered_decomposition} for the precise formulation. The coupling operators $M_{i,j}$ encode interactions between adjacent layers. Their construction and the \emph{global coupling condition}, which requires finite-row independence in the directions of extension, are detailed in Section~\ref{sec:linear}.

	Eidelheit's theorem connects this condition with the surjectivity needed to extend prescribed initial-layer data. We obtain the following result.
	
	\begin{theorem}\label{thm:mainSur}
		Let $G$ be a connected, locally finite, infinite graph equipped with a layered decomposition satisfying the global coupling condition, and let
		\[
			\mathcal S_f:=\{u\in\mathbb R^V:\Delta u(x)=f(x,u(x))\text{ for every }x\in V\}.
		\]
		For every function $f:V\times\mathbb R\to\mathbb R$, the following assertions hold.
		\begin{enumerate}
			\item If the decomposition is $\mathbb N$-layered, the restriction map $\mathcal S_f\to\mathbb R^{V_0}$ is surjective.
			\item If the decomposition is $\mathbb Z$-layered, the restriction map $\mathcal S_f\to\mathbb R^{V_{-1}}\times\mathbb R^{V_0}$ is surjective.
		\end{enumerate}
		In particular, the equation has infinitely many solutions.
	\end{theorem}
	
	Finite-row independence admits an exact finite-dimensional characterization by nonzero full-row-size minors; unique perfect matchings provide a transparent sufficient condition.
	
	\begin{theorem}\label{thm:mainFinite}
		A coupling operator $M_{i,j}$ satisfies the finite-row independence condition if and only if, for every finite subset $U \subset V_i$, there exists a subset $K \subset V_j$ with $|K| = |U|$ such that the corresponding bipartite adjacency submatrix $M_{U,K}$ has nonzero determinant.

		Finite-row independence follows if, for every such $U$, one can choose $K$ so that the induced bipartite graph between $U$ and $K$ has a unique perfect matching. A stronger sufficient condition is the existence of an injective map $\phi:V_i\to V_j$ such that $\phi(x)$ is the unique neighbor of $x$ in $V_j$ for every $x\in V_i$.
	\end{theorem}
	
	These criteria apply to locally finite infinite trees of minimum degree at least two, integer and planar lattices, and several Cayley graphs, including the discrete Heisenberg group $\mathcal{H}_3(\mathbb Z)$. On $\mathbb Z^d$, Proposition~\ref{prop:lattice_bijection} gives more: restriction is a bijection from the solution set onto the space of data on any two adjacent layers. Hence $\Delta u=f(x,u(x))$ has infinitely many global solutions for every $f$, including nonlinearities of Yamabe, Kazdan--Warner, and Liouville type. Section~\ref{subsec:semidirect_products} then treats broader classes of semidirect-product Cayley graphs, with solvability for left-orderable base groups and determinant criteria for finite and split mixed base groups.
	
	\section{Preliminaries}
	\label{sec:preliminaries}
	
	\subsection{Infinite Graphs and the Laplacian}
	
	Let $G = (V, E)$ be a connected, locally finite, undirected simple graph with a countably infinite vertex set $V$. Vertices $x, y \in V$ are adjacent, denoted by $x \sim y$, if $\{x, y\} \in E$. Local finiteness implies the degree of any vertex $x$, denoted by $\deg(x) = |\{y \in V : y \sim x\}|$, is finite.
	
	For a subset $U\subset V$, its vertex neighborhood is defined by
	\[
		N(U):=\{y\in V:\text{there exists }x\in U\text{ such that }y\sim x\}.
	\]
	In particular, local finiteness implies that $N(U)$ is finite whenever $U$ is finite.
	
	Equip the vector space $\mathbb{R}^V$ with the topology of pointwise convergence. It is then a Fr\'{e}chet space, and its topology is characterized by the family of seminorms
	\begin{align}\label{eq:seminorms}
		p_K(u) := \sum_{x \in K} |u(x)|, \quad u \in \mathbb{R}^V,
	\end{align}
	where $K \subset V$ ranges over all finite subsets of vertices. Under this topology, the dual space is identified with the space of finitely supported functions $v: V \to \mathbb{R}$.
	
	The unnormalized combinatorial Laplacian $\Delta: \mathbb{R}^V \to \mathbb{R}^V$ is defined by
	\begin{equation}
		\label{eq:laplacian}
		\Delta u(x) = \sum_{y \sim x} (u(y) - u(x)).
	\end{equation}
	Local finiteness of $G$ ensures that the sum in~\eqref{eq:laplacian} contains only finitely many terms for each $x \in V$. Thus, $\Delta$ is a continuous linear operator on $\mathbb{R}^V$; it has finite hopping range and satisfies the discrete maximum principle~\cite{Kalmes2016}.
	
	\subsection{Layered Decomposition}
	
	To formulate the extension argument, we partition the vertex set into layers so that the equation on each layer involves only that layer and its immediate neighbors. The following definition specifies this structure.
	
	\begin{definition}
		\label{def:layered_decomposition}
		A graph $G = (V, E)$ admits an \textit{$\mathcal{I}$-layered decomposition} over an index set $\mathcal{I} \in \{\mathbb{N}, \mathbb{Z}\}$ if its vertex set can be partitioned as $V = \bigsqcup_{n \in \mathcal{I}} V_n$, where each layer $V_n$ is a nonempty subset of $V$, such that for any $x \in V_n$, its neighbors $y \sim x$ belong to $V_{n-1} \cup V_n \cup V_{n+1}$ (with $V_{-1} = \emptyset$ for $\mathcal{I} = \mathbb{N}$).
	\end{definition}
	
	No restriction is imposed on the cardinality of an individual layer: the sets $V_n$ may be finite or countably infinite, and their cardinalities need not be uniform.
	
	\begin{example}
		\label{ex:layered_graphs}
		The following examples illustrate layered decompositions of infinite graphs.
		\begin{enumerate}
			\item \textit{Integer lattices.}
			Write $\mathbb Z^d\cong\mathbb Z\times\mathbb Z^{d-1}$ and set $V_n=\{(n,x'):x'\in\mathbb Z^{d-1}\}$. A nearest-neighbor edge either remains in $V_n$ or changes the first coordinate by one, so these layers give a $\mathbb Z$-layered decomposition.
			
			\item \textit{Distance spheres.}
			Fix a root $x_0\in V$ and define $V_n=\{x\in V:d(x,x_0)=n\}$ for $n\in\mathbb N$. If $x\sim y$, then $|d(x,x_0)-d(y,x_0)|\leq1$, so the distance spheres form an $\mathbb N$-layered decomposition. This observation provides a general source of layered decompositions, although it does not by itself imply the global coupling condition.
			
			\item \textit{Infinite trees.}
			For a rooted infinite tree, define the root to have depth zero and every other vertex to have depth one greater than that of its parent. If $V_n$ is the set of vertices of depth $n$, then $(V_n)_{n\in\mathbb N}$ is an $\mathbb N$-layered decomposition, and every edge joins consecutive layers.
			
			\Needspace{16\baselineskip}
			\item \textit{The discrete Heisenberg group.}
			The group $\mathcal{H}_3(\mathbb{Z})$ consists of $3 \times 3$ upper triangular integer matrices with unit diagonal. We use coordinates $(x,y,z)\in\mathbb Z^3$ through
			\[
				(x,y,z)
				\longleftrightarrow
				\begin{pmatrix}
					1 & x & z \\
					0 & 1 & y \\
					0 & 0 & 1
				\end{pmatrix}.
			\]
			The multiplication law is therefore given directly by matrix multiplication:
			\[
				\begin{pmatrix}
					1 & x_1 & z_1 \\
					0 & 1 & y_1 \\
					0 & 0 & 1
				\end{pmatrix}
				\begin{pmatrix}
					1 & x_2 & z_2 \\
					0 & 1 & y_2 \\
					0 & 0 & 1
				\end{pmatrix}
				=
				\begin{pmatrix}
					1 & x_1+x_2 & z_1+z_2+x_1y_2 \\
					0 & 1 & y_1+y_2 \\
					0 & 0 & 1
				\end{pmatrix}.
			\]
			Set $X=(1,0,0)$, $Y=(0,1,0)$, and $Z=(0,0,1)$, and take the symmetric generating set $S = \{X^{\pm 1}, Y^{\pm 1}, Z^{\pm 1}\}$. With $V_n=\{(x,n,z):x,z\in\mathbb Z\}$, right multiplication by $X^{\pm1}$ and $Z^{\pm1}$ preserves the layer, whereas right multiplication by $Y^{\pm1}$ moves to the adjacent layer.
		\end{enumerate}
	\end{example}

	\section{Coupling Operators}
	\label{sec:linear}
	
	For a graph admitting an $\mathcal{I}$-layered decomposition $V = \bigsqcup_{n \in \mathcal{I}} V_n$, write $u_n=u|_{V_n}\in\mathbb R^{V_n}$ for the restriction of $u\in\mathbb R^V$ to the $n$th layer. The space $\mathbb{R}^{V_n}$ is equipped with the topology of pointwise convergence, which reduces to the standard Euclidean topology when $V_n$ is finite.

	The first index of $M_{i,j}$ is the layer on which the equation is evaluated, and the second is the layer containing the input. Define the following continuous linear operators:
	\begin{itemize}
		\item For $|i-j|=1$, the \textit{coupling operator} $M_{i,j}:\mathbb R^{V_j}\to\mathbb R^{V_i}$ is defined by
		\[
			(M_{i,j}v)(x)=\sum_{\substack{y\in V_j\\y\sim x}}v(y),
			\qquad x\in V_i.
		\]
		It is called forward when $j=i+1$ and backward when $j=i-1$.
		\item The \textit{diagonal Laplacian block} $\Delta_n: \mathbb{R}^{V_n} \to \mathbb{R}^{V_n}$ is defined for $w \in \mathbb{R}^{V_n}$ and $x \in V_n$ by $(\Delta_n w)(x) = \sum_{y \in V_n, y \sim x} w(y) - \deg(x)w(x)$.
	\end{itemize}

	The term \emph{coupling operator} refers to the linear operator itself; its vertex-indexed coefficient array is the corresponding \emph{coupling matrix}. For finite $U\subset V_i$ and $K\subset V_j$, we write $M_{U,K}$ for the submatrix with rows indexed by $U$ and columns by $K$. The diagonal term in $\Delta_n$ uses the global degree $\deg(x)$ of $x$ in $G$, rather than the degree in the subgraph induced by $V_n$. Local finiteness ensures that these operators are continuous and that every row of each coupling matrix has finite support.

	By Definition~\ref{def:layered_decomposition}, the equation on $V_n$ depends only on $u_{n-1}$, $u_n$, and $u_{n+1}$. For a nonlinearity $f:V\times\mathbb R\to\mathbb R$, it takes the form
	\begin{equation}
		\label{eq:layer_decomposition}
		(\Delta_n u_n)(x) + \bigl(M_{n,n+1} u_{n+1}\bigr)(x) + \bigl(M_{n,n-1} u_{n-1}\bigr)(x) = f(x, u_n(x)).
	\end{equation}

	\subsection{Finite-Row Independence}
	
	For an $\mathbb N$-layered graph one prescribes $u_0$, whereas for a $\mathbb Z$-layered graph one prescribes two adjacent layers, say $(u_{-1},u_0)$. Isolating the next unknown layer transforms the global equation into a sequential extension problem.
	
	When $\mathcal{I}=\mathbb N$, with $V_{-1}=\emptyset$, the forward extension relation is
	\begin{equation}
		\label{eq:forward_extension}
		M_{n,n+1} u_{n+1} = f(\cdot, u_n) - \Delta_n u_n - M_{n,n-1} u_{n-1}.
	\end{equation}
	
	The $\mathbb{N}$-layered system is solvable for arbitrary initial data and nonlinearity $f$ provided each $M_{n,n+1}$ is surjective.
	
	\begin{definition}
		\label{def:finite_surjectivity}
		A coupling operator $M_{i,j}: \mathbb{R}^{V_j} \to \mathbb{R}^{V_i}$ satisfies the \textit{finite-row independence condition} if for every finite subset $U \subset V_i$, the restricted operator
		\[ \pi_U \circ M_{i,j}: \mathbb{R}^{V_j} \to \mathbb{R}^U \]
		is surjective, equivalently has rank $|U|$, where $\pi_U$ denotes coordinate restriction to $U$. Equivalently, the coordinate functionals $v\mapsto(M_{i,j}v)(x)$, $x\in U$, are linearly independent.
	\end{definition}
	
	The topology of pointwise convergence and the local finiteness of $G$ ensure that the coupling operators are continuous linear operators between Fr\'{e}chet spaces. We use the following form of Eidelheit's theorem~\cite{Eidelheit1936,Kalmes2016}.
	
	\begin{lemma}
		\label{lem:eidelheit_strict}
		Let $\mathbb K\in\{\mathbb R,\mathbb C\}$, let $E$ be a Fr\'{e}chet space over $\mathbb K$ equipped with an increasing fundamental system of seminorms $(p_k)_{k \in \mathbb{N}}$, and let $\omega_{\mathbb K}=\mathbb K^{\mathbb N}$. A continuous $\mathbb K$-linear map $A:E\to\omega_{\mathbb K}$ is surjective if and only if:
		\begin{enumerate}
			\item The sequence of coordinate functionals $(A_j)_{j \in \mathbb{N}}$ defined by $A_j = \pi_j \circ A \in E'$, where $\pi_j$ is the $j$th coordinate projection, is linearly independent.
			\item For every $k \in \mathbb{N}$, the following subspace of the dual space $E'$ is finite-dimensional:
			\begin{equation*}
				\dim \left( \left\{ \phi \in E' : \exists c > 0,\ \lvert\phi(x)\rvert \le c\,p_k(x)\ \text{for all }x\in E \right\} \cap \operatorname{span}\{ A_j : j \in \mathbb{N} \} \right) < \infty.
			\end{equation*}
		\end{enumerate}
	\end{lemma}
	
	\begin{proposition}
		\label{prop:eig_solvability}
		The coupling operator $M_{i,j}$ is surjective if and only if it satisfies the finite-row independence condition.
	\end{proposition}
	
	\begin{proof}
		If $V_i$ is finite, local finiteness confines all nonzero entries in its finitely many rows to a finite set of columns, so this is the ordinary finite-dimensional rank criterion.

		Now suppose that $V_i$ is infinite. If $V_j$ is finite, neither condition can hold: for $|U|>|V_j|$, the operator $\pi_U\circ M_{i,j}$ cannot be surjective, and a finite-dimensional domain cannot map onto $\mathbb R^{V_i}$.
		
		It remains to consider the case in which both layers are infinite. Enumerate $V_i$ to identify $\mathbb R^{V_i}$ with $\omega_{\mathbb R}$, and choose an increasing finite exhaustion $K_k\nearrow V_j$. Equip $\mathbb R^{V_j}$ with the seminorms
		\[
			p_k(v):=\sum_{y\in K_k}|v(y)|.
		\]
		A functional bounded by $c p_k$ vanishes on every function supported outside $K_k$ and lies in the finite-dimensional space $\operatorname{span}\{\delta_y:y\in K_k\}$, where $\delta_y(v):=v(y)$. Condition~(2) of Lemma~\ref{lem:eidelheit_strict} is therefore automatic. The coordinate functionals of $M_{i,j}$ are its rows, and every relation among them involves only finitely many elements of $V_i$. Condition~(1) is precisely finite-row independence. Eidelheit's theorem gives the equivalence.
	\end{proof}
	\begin{remark}
		Kalmes~\cite{Kalmes2016} applies Eidelheit's theorem to the global Poisson operator; Koberstein and Schmidt~\cite{KobersteinSchmidt2020} express the finite-hopping criterion through transpose injectivity. Proposition~\ref{prop:eig_solvability} applies the same principle to rectangular interlayer coupling operators whose right-hand sides may depend nonlinearly on previously constructed layers.

		The same surjectivity can also be viewed through finite affine systems, as in the projective-limit approach of~\cite{Dodziuk2012}. For a fixed right-hand side and an increasing finite exhaustion of $V_i$, finite-row independence makes each finite system solvable on its finite set of active coordinates. At every fixed stage, the images of the later affine solution spaces form a decreasing sequence of nonempty affine subspaces and stabilize by finite dimensionality. The Mittag--Leffler condition then gives a compatible family of solutions; unused coordinates in $V_j$ can be chosen freely.
	\end{remark}

	\begin{definition}
		\label{def:global_coupling}
		A graph equipped with a chosen layered decomposition satisfies the \emph{global coupling condition} relative to this decomposition if $M_{n,n+1}$ has finite-row independence for every $n\geq0$. In the $\mathbb Z$-layered case, we additionally require finite-row independence of $M_{n,n-1}$ for every $n\leq-1$.
	\end{definition}

	\begin{proof}[Proof of Theorem~\ref{thm:mainSur}]
		For an $\mathbb N$-layered graph, prescribe $u_0\in\mathbb R^{V_0}$ and set $V_{-1}=\emptyset$. Given $u_{n-1}$ and $u_n$, Proposition~\ref{prop:eig_solvability} provides $u_{n+1}$ solving~\eqref{eq:forward_extension}. This enforces the equation on $V_n$ without altering the equations on earlier layers. Induction yields a global solution.

		For a $\mathbb Z$-layered graph, prescribe $(u_{-1},u_0)$. Surjectivity of the coupling operators in Definition~\ref{def:global_coupling} provides $u_1$ from the equation on $V_0$ and $u_{-2}$ from the equation on $V_{-1}$. Iterating the forward coupling operators for $n\geq0$ and the backward coupling operators for $n\leq-1$ extends the solution in both directions. In either case the prescribed data were arbitrary, so the relevant restriction map is surjective. Each target contains a copy of $\mathbb R$, and the solution set is infinite.
	\end{proof}
	\subsection{Finite Minors and Matchings}
	In vertex-indexed coordinates, the coupling operator $M_{i,j}$ has a possibly infinite row-finite $0$--$1$ matrix representation, also denoted by $M_{i,j}$. Its rows are indexed by $V_i$, its columns by $V_j$, and its entries are
	\[
	(M_{i,j})_{x,y} =
	\begin{cases}
		1, & \text{if } x \sim y, \\
		0, & \text{otherwise.}
	\end{cases}
	\]
	Under the layered partition, these matrices are the off-diagonal blocks of the block-tridiagonal adjacency operator. If $U\subset V_i$ is finite, local finiteness confines the nonzero columns of the rows indexed by $U$ to the finite set $N(U)\cap V_j$. After zero columns are removed, the row restriction $\pi_U\circ M_{i,j}$ is represented by the finite active matrix $M_{i,j}^U$, with rows indexed by $U$ and columns by $N(U)\cap V_j$. Finite-row independence is equivalent to full row rank of $M_{i,j}^U$ for every such $U$. Figure~\ref{fig:coupling_matrix} depicts this reduction.

	\begin{figure}[htbp]
		\centering
			\begin{tikzpicture}[>=stealth, node distance=1.5cm, scale=0.8]
			
			\node at (1.5, 3.5) {\small Layered Graph};
			
			\node[circle, draw, fill=gray!20, inner sep=2pt, label=left:{$x_1$}] (u1) at (0, 1.2) {};
			\node[circle, draw, fill=gray!20, inner sep=2pt, label=left:{$x_2$}] (u2) at (0, 0) {};
			\node[circle, draw, fill=gray!20, inner sep=2pt, label=left:{$x_3$}] (u3) at (0, -1.2) {};
			\draw[gray!50] (0, 2) -- (0, -2);
			\node at (0, -2.5) {\small $\vdots$};
			\node at (0, 2.5) {\small $\vdots$};
			\node at (0, -3.8) {\small $V_i$};
			
			\node[circle, draw, fill=gray!10, inner sep=2pt, label=right:{$y_1$}] (v1) at (3, 1.8) {};
			\node[circle, draw, fill=gray!10, inner sep=2pt, label=right:{$y_2$}] (v2) at (3, 0.6) {};
			\node[circle, draw, fill=gray!10, inner sep=2pt, label=right:{$y_3$}] (v3) at (3, -0.6) {};
			\node[circle, draw, fill=gray!10, inner sep=2pt, label=right:{$y_4$}] (v4) at (3, -1.8) {};
			\draw[gray!50] (3, 2.5) -- (3, -2.5);
			\node at (3, -3) {\small $\vdots$};
			\node at (3, 3) {\small $\vdots$};
			\node at (3, -3.8) {\small $V_j$};
			
			\draw[thick, gray!60] (u1) -- (v1);
			\draw[thick, gray!60] (u1) -- (v2);
			\draw[thick, gray!60] (u2) -- (v2);
			\draw[thick, gray!60] (u2) -- (v3);
			\draw[thick, gray!60] (u2) -- (v4);
			\draw[thick, gray!60] (u3) -- (v1);
			\draw[thick, gray!60] (u3) -- (v4);
			
			\draw[thick,gray!60,dashed] (v1) -- (1.5,2);
			\draw[thick,gray!60,dashed] (v4) -- (1.5,-2);
			
			\begin{scope}[shift={(8,0)}]
				\node at (1.5, 3.5) {\small Matrix of $M_{i,j}$};
				
				\node at (1.5, 0) {
					$\begin{pmatrix} 
						\ddots&*&*&*&*&\\
						&1 & 1 & 0 & 0 & \\ 
						&0 & 1 & 1 & 1 & \\ 
						&1 & 0 & 0 & 1 &\\
						&*&*&*&*&\ddots
					\end{pmatrix}$
				};
				
				\node[rotate=90, scale=0.8] at (-1.2, 0) {\small Rows in $U$};
				\node[scale=0.8] at (1.5, 1.8) {\small Cols in $N(U)\cap V_j$};
			\end{scope}
			
			\draw[->, thick] (3.8, 0) -- (5.8, 0);
			
		\end{tikzpicture}
		\caption{A finite row set $U\subset V_i$ and the corresponding finite submatrix of the matrix representing $M_{i,j}$.}
		\label{fig:coupling_matrix}
	\end{figure}
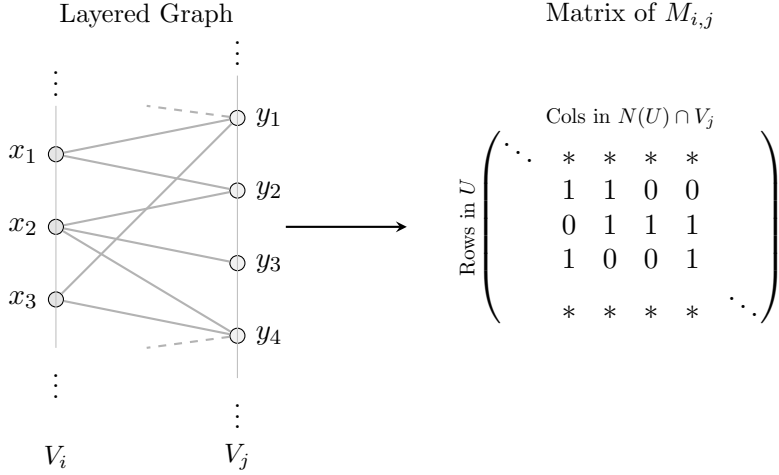

	\begin{proof}[Proof of Theorem~\ref{thm:mainFinite}]
		For a fixed finite $U\subset V_i$, the rows indexed by $U$ are independent if and only if the finite matrix $M_{i,j}^U$ has rank $|U|$. By the standard maximal-minor criterion, this is equivalent to the existence of $K\subset N(U)\cap V_j$ with $|K|=|U|$ and $\det M_{U,K}\neq0$. Requiring this for every finite $U$ proves the stated equivalence.
		If the induced graph between $U$ and $K$ has a unique perfect matching, exactly one term in the Leibniz expansion of $\det M_{U,K}$ is nonzero. Hence $\det M_{U,K}=\pm1$.
		
		For the final statement, assume that $N(\{x\})\cap V_j=\{\phi(x)\}$ for every $x\in V_i$, where $\phi:V_i\to V_j$ is injective. Then $(M_{i,j}v)(x)=v(\phi(x))$: the vertex injection and the coupling operator act in opposite directions. For a finite $U\subset V_i$, set $K=\phi(U)$. The induced bipartite graph on $U\sqcup K$ consists exactly of the disjoint edges $(x,\phi(x))$ and hence has a unique perfect matching. Finite-row independence follows.
	\end{proof}

	An injective choice of adjacent targets alone need not suffice: additional cross-layer edges can create perfect matchings whose determinant terms cancel. The final criterion in Theorem~\ref{thm:mainFinite} imposes the stronger singleton-neighbor condition
	\[
		N(\{x\})\cap V_{n+1}=\{\phi_n(x)\},
		\qquad \phi_n:V_n\to V_{n+1}\ \text{injective}.
	\]
	This condition guarantees forward extension, and the analogous condition toward $V_{n-1}$ guarantees backward extension.
	
	\section{Applications}
	\label{sec:applications}
	
	\subsection{Standard Examples}
	\label{subsec:standard_geometries}
	
	Integer lattices exhibit a stronger conclusion than mere surjectivity. For $\mathbb Z^d\cong\mathbb Z\times\mathbb Z^{d-1}$, every interlayer edge joins $(n,x')$ to $(n\pm1,x')$, so under the natural layer identifications each coupling operator $M_{n,n\pm1}$ is the identity.

	\begin{proposition}
		\label{prop:lattice_bijection}
		Let $d\geq1$ and $f:\mathbb Z^d\times\mathbb R\to\mathbb R$, and let
		\[
			\mathcal S_f:=\bigl\{u\in\mathbb R^{\mathbb Z^d}:\Delta u(x)=f(x,u(x))
			\text{ for every }x\in\mathbb Z^d\bigr\}
		\]
		be the solution set. For the decomposition $\mathbb Z^d\cong\mathbb Z\times\mathbb Z^{d-1}$ and any fixed $k\in\mathbb Z$, the restriction map
		\[ \mathcal{R}_k : \mathcal{S}_f \to \mathbb{R}^{V_k} \times \mathbb{R}^{V_{k+1}}, \quad u \mapsto (u|_{V_k}, u|_{V_{k+1}}) \]
		is a bijection.
	\end{proposition}

	\begin{proof}
		Given $(a,b)\in\mathbb R^{V_k}\times\mathbb R^{V_{k+1}}$, set $u_k=a$ and $u_{k+1}=b$. The equations on $V_{k+1}$ and $V_k$ determine $u_{k+2}$ and $u_{k-1}$, respectively. Iteration in both directions yields a unique global solution with the prescribed restrictions.
	\end{proof}

	\Needspace{12\baselineskip}
	\begin{proof}[Proof of Corollary~\ref{cor:exponential_lattice}]
		Proposition~\ref{prop:lattice_bijection} gives infinitely many solutions. Assume that a solution $u$ has finite energy:
		\[
			\sum_{x \in \mathbb{Z}^2} e^{u(x)} < \infty .
		\]
		Every superlevel set $\{x\in\mathbb Z^2:u(x)\geq M\}$ is finite. For any $x_*\in\mathbb Z^2$, the finite nonempty set $\{x:u(x)\geq u(x_*)\}$ therefore contains a global maximizer $x_0$. Hence
		\[
			\Delta u(x_0)=\sum_{y\sim x_0}\bigl(u(y)-u(x_0)\bigr)\leq0,
		\]
		contrary to $\Delta u(x_0)=e^{u(x_0)}>0$.
	\end{proof}

	\begin{example}[Antitrees]
		Consider an $\mathbb N$-layered graph with finite nonempty layers in which every vertex of $V_n$ is adjacent to every vertex of $V_{n+1}$. If $|V_n|\geq2$, all rows of $M_{n,n+1}$ are identical. Its coupling matrix has rank one, so the operator fails finite-row independence.
	\end{example}
	
	\begin{example}[Leafless trees]
		Let $\mathbb{T}$ be a locally finite infinite tree satisfying $\deg(x)\geq2$ for every vertex $x$. Choose a root and let $V_n$ be the set of vertices of depth $n$. Every vertex has at least one child, and every child has a unique parent.

		For finite $U\subset V_n$, choose one child $y_x$ of each $x\in U$. The induced bipartite graph on $U\sqcup\{y_x:x\in U\}$ consists of the edges $(x,y_x)$. Hence $M_{n,n+1}$ has finite-row independence.
	\end{example}
	
	\begin{example}[Planar periodic graphs]
		The following planar periodic graphs also satisfy the global coupling condition.
		\begin{itemize}
			\item \textit{The hexagonal graph.} Embed the infinite hexagonal graph into $\mathbb{Z}^2$ via the brick-wall model, with horizontal edges $(x,y)\sim(x+1,y)$ and vertical edges $(x,y)\sim(x,y+1)$ whenever $x+y$ is even. Define $V_n=\{(n,y):y\in\mathbb Z\}$. The vertical edges remain within $V_n$; the horizontal edges identify $V_n$ bijectively with both $V_{n-1}$ and $V_{n+1}$. Theorem~\ref{thm:mainFinite} applies in both directions.
			
			\item \textit{The triangular lattice.} Use the model with vertex set $\mathbb Z^2$ in which two vertices are adjacent when their difference is one of $\pm(1,0)$, $\pm(0,1)$, or $\pm(1,1)$, and take the layers $V_n=\{(n,y):y\in\mathbb Z\}$. For finite $U=\{(n,y_1),\ldots,(n,y_m)\}$ with $y_1<\cdots<y_m$, take $K=U+(1,0)$. In these orderings the restricted adjacency matrix is upper triangular with diagonal entries one, hence has a unique perfect matching. The choice $K=U-(1,0)$ gives the backward condition.
		\end{itemize}
	\end{example}
	These four graphs are shown in Figure~\ref{fig:standard_geometries}.

	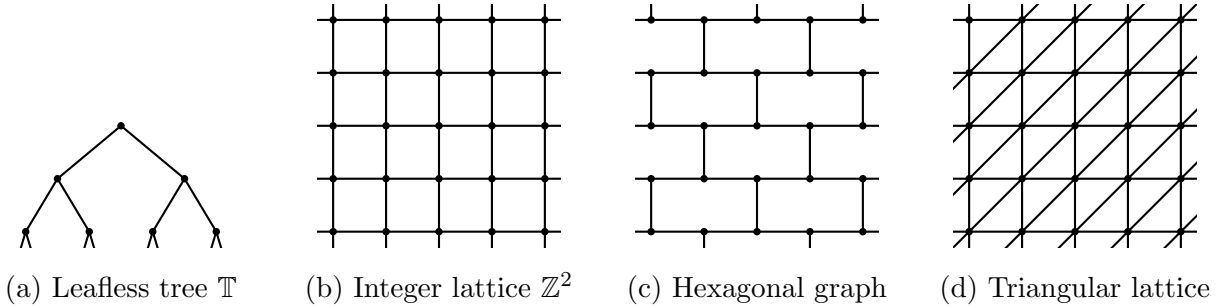
\begin{figure}[htbp]
		\centering
		\begin{minipage}[t]{0.24\linewidth}
			\centering
			\begin{tikzpicture}[scale=0.7]
				\clip (-2.3, -2.3) rectangle (2.3, 2.3);
				\draw[thick] (0,0) -- (-1.2,-1) -- (-1.8,-2) -- (-2.1,-3);
				\draw[thick] (-1.8,-2) -- (-1.5,-3);
				\draw[thick] (-1.2,-1) -- (-0.6,-2) -- (-0.9,-3);
				\draw[thick] (-0.6,-2) -- (-0.3,-3);
				\draw[thick] (0,0) -- (1.2,-1) -- (0.6,-2) -- (0.3,-3);
				\draw[thick] (0.6,-2) -- (0.9,-3);
				\draw[thick] (1.2,-1) -- (1.8,-2) -- (1.5,-3);
				\draw[thick] (1.8,-2) -- (2.1,-3);
				\fill (0,0) circle (2pt);
				\fill (-1.2,-1) circle (2pt);
				\fill (1.2,-1) circle (2pt);
				\fill (-1.8,-2) circle (2pt);
				\fill (-0.6,-2) circle (2pt);
				\fill (0.6,-2) circle (2pt);
				\fill (1.8,-2) circle (2pt);
			\end{tikzpicture}
			\\[0.5em] (a) Leafless tree $\mathbb T$
		\end{minipage}\hfill
		\begin{minipage}[t]{0.24\linewidth}
			\centering
			\begin{tikzpicture}[scale=0.7]
				\clip (-2.3, -2.3) rectangle (2.3, 2.3);
				\foreach \i in {-3,...,3} {
					\draw[thick] (\i, -3) -- (\i, 3);
					\draw[thick] (-3, \i) -- (3, \i);
				}
				\foreach \x in {-2,...,2} {
					\foreach \y in {-2,...,2} {
						\fill (\x, \y) circle (2pt);
					}
				}
			\end{tikzpicture}
			\\[0.5em] (b) Integer lattice $\mathbb Z^2$
		\end{minipage}\hfill
		\begin{minipage}[t]{0.24\linewidth}
			\centering
			\begin{tikzpicture}[scale=0.7]
				\clip (-2.3, -2.3) rectangle (2.3, 2.3);
				\foreach \y in {-3,...,3} {
					\draw[thick] (-3, \y) -- (3, \y);
				}
				\foreach \x in {-3,...,3} {
					\foreach \y in {-3,...,3} {
						\pgfmathtruncatemacro{\sumxy}{\x+\y}
						\ifodd\sumxy\else
						\draw[thick] (\x, \y) -- (\x, \y+1);
						\fi
					}
				}
				\foreach \x in {-2,...,2} {
					\foreach \y in {-2,...,2} {
						\fill (\x, \y) circle (2pt);
					}
				}
			\end{tikzpicture}
			\\[0.5em] (c) Hexagonal graph
		\end{minipage}\hfill
		\begin{minipage}[t]{0.24\linewidth}
			\centering
			\begin{tikzpicture}[scale=0.7]
				\clip (-2.3, -2.3) rectangle (2.3, 2.3);
				\foreach \i in {-3,...,3} {
					\draw[thick] (\i, -3) -- (\i, 3);
					\draw[thick] (-3, \i) -- (3, \i);
				}
				\foreach \i in {-6,...,3} {
					\draw[thick] (\i, -3) -- (\i+6, 3);
				}
				\foreach \x in {-2,...,2} {
					\foreach \y in {-2,...,2} {
						\fill (\x, \y) circle (2pt);
					}
				}
			\end{tikzpicture}
			\\[0.5em] (d) Triangular lattice
		\end{minipage}
			\caption{Graph models for a leafless tree, the integer lattice, the hexagonal graph, and the triangular lattice.}
		\label{fig:standard_geometries}
	\end{figure}
	\FloatBarrier
	
		\Needspace{8\baselineskip}
		\begin{example}[Discrete Heisenberg group]
			\label{ex:heisenberg_coupling}
			For the $y$-layering fixed in Example~\ref{ex:layered_graphs}, $Y^{\pm1}$ are the only generators joining distinct layers. Right multiplication gives bijections
			\[
				\phi_n^+(x,n,z)=(x,n+1,z+x),
				\qquad
				\phi_n^-(x,n,z)=(x,n-1,z-x)
			\]
			from $V_n$ onto $V_{n+1}$ and $V_{n-1}$, respectively. The unique-neighbor criterion in Theorem~\ref{thm:mainFinite} therefore applies in both directions. Figure~\ref{fig:heisenberg_shear} illustrates the resulting shear.
		\end{example}

		\begin{figure}[!htbp]
			\centering
			\begin{tikzpicture}[x={(0.6cm, -0.4cm)}, y={(3.6cm, 0cm)}, z={(0cm, 0.9cm)}, scale=0.85, >=stealth]
				
				\tikzset{
					layer plane/.style={fill=gray!5, draw=gray!30, opacity=0.8},
					grid line/.style={gray!30, very thin},
					grid dot/.style={circle, fill=gray!30, inner sep=0.8pt},
					active edge/.style={->, thick, shorten >= 1.5pt, shorten <= 1.5pt},
					node blue/.style={circle, fill=blue!70, inner sep=1.5pt},
					node dark/.style={circle, fill=gray!80!black, inner sep=1.5pt},
					node red/.style={circle, fill=red!70, inner sep=1.5pt},
					edge blue/.style={active edge, blue!70},
					edge dark/.style={active edge, gray!80!black},
					edge red/.style={active edge, red!70}
				}
				
				\foreach \layer in {0,1,2} {
					\draw[layer plane] (-1.8, \layer, -4) -- (1.8, \layer, -4) -- (1.8, \layer, 4) -- (-1.8, \layer, 4) -- cycle;
					
					\node[text=gray!80, anchor=south, font=\footnotesize] at (0, \layer, 4.5) {$V_{\layer}$ ($y=\layer$)};
					
					\foreach \x in {-1,0,1} {
						\draw[grid line] (\x, \layer, -3.5) -- (\x, \layer, 3.5);
					}
					\foreach \z in {-3,-2,-1,0,1,2,3} {
						\draw[grid line] (-1.5, \layer, \z) -- (1.5, \layer, \z);
					}
					\foreach \x in {-1,0,1} {
						\foreach \z in {-3,-2,-1,0,1,2,3} {
							\node[grid dot] at (\x, \layer, \z) {};
						}
					}
				}
				
				\draw[thick, orange, rounded corners=2pt, dashed] 
				(-1.2, 0, -1.2) -- (1.2, 0, -1.2) -- (1.2, 0, 1.2) -- (-1.2, 0, 1.2) -- cycle;
				\node[text=orange, anchor=north, font=\footnotesize] at (0, 0, -1.8) {Subset $U$};
				
				\draw[thick, orange, rounded corners=2pt, dashed] 
				(-1.2, 1, -2.4) -- (1.2, 1, 0) -- (1.2, 1, 2.4) -- (-1.2, 1, 0) -- cycle;
				\node[text=orange, anchor=north, font=\footnotesize] at (0, 1, -2.2) {Subset $K = UY$};
				
				\draw[thick, orange!50, rounded corners=2pt, dashed] 
				(-1.2, 2, -3.6) -- (1.2, 2, 1.2) -- (1.2, 2, 3.6) -- (-1.2, 2, -1.2) -- cycle;
				
				\foreach \zbase in {-1,0,1} {
					\node[node blue] (b0) at (-1, 0, \zbase) {};
					\node[node blue] (b1) at (-1, 1, \zbase-1) {};
					\node[node blue] (b2) at (-1, 2, \zbase-2) {};
					\draw[edge blue] (b0) -- (b1);
					\draw[edge blue] (b1) -- (b2);
					
					\node[node dark] (d0) at (0, 0, \zbase) {};
					\node[node dark] (d1) at (0, 1, \zbase) {};
					\node[node dark] (d2) at (0, 2, \zbase) {};
					\draw[edge dark] (d0) -- (d1);
					\draw[edge dark] (d1) -- (d2);
					
					\node[node red] (r0) at (1, 0, \zbase) {};
					\node[node red] (r1) at (1, 1, \zbase+1) {};
					\node[node red] (r2) at (1, 2, \zbase+2) {};
					\draw[edge red] (r0) -- (r1);
					\draw[edge red] (r1) -- (r2);
				}
				
				\node[fill=white, draw=gray!40, rounded corners, align=left, font=\scriptsize] 
				at (-1.8, 0, 5) [anchor=south west] {
					\textbf{Right multiplication by $Y$}\\
					$(x,n,z) \xrightarrow{\;Y\;} (x, n+1, z+x)$ \\[1.5mm]
					\textcolor{red!70}{\rule{4mm}{1.5mm}} $\ x=1 \implies z \mapsto z+1$ (Upward shear)\\
					\textcolor{gray!80!black}{\rule{4mm}{1.5mm}} $\ x=0 \implies z \mapsto z$ (No shear)\\
					\textcolor{blue!70}{\rule{4mm}{1.5mm}} $\ x=-1 \implies z \mapsto z-1$ (Downward shear)
				};
				
				\begin{scope}[shift={(-1.8, 2.3, 5)}]
					\draw[->, thick, gray!80!black] (0,0,0) -- (1,0,0) node[anchor=north east, font=\scriptsize] {$x$};
					\draw[->, thick, gray!80!black] (0,0,0) -- (0,0.3,0) node[anchor=west, font=\scriptsize] {$y$};
					\draw[->, thick, gray!80!black] (0,0,0) -- (0,0,1) node[anchor=south, font=\scriptsize] {$z$};
				\end{scope}
				
			\end{tikzpicture}
			
			\caption{The $y$-layering of $\mathcal H_3(\mathbb Z)$. Right multiplication by $Y$ sends $(x,n,z)$ to $(x,n+1,z+x)$ and maps $U\subset V_0$ bijectively onto $UY\subset V_1$, realizing the unique-neighbor criterion of Theorem~\ref{thm:mainFinite}.}
			\label{fig:heisenberg_shear}
		\end{figure}
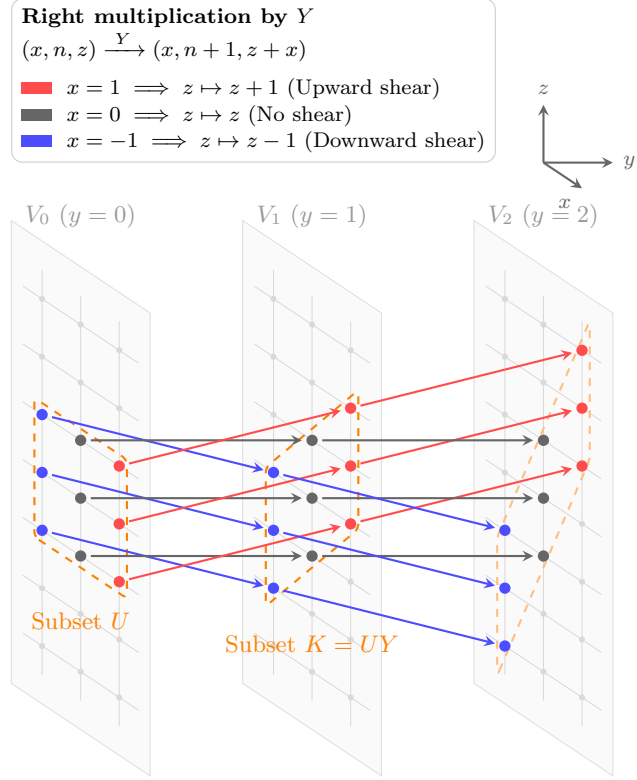
		\FloatBarrier
		\Needspace{10\baselineskip}
	\subsection{Semidirect-Product Cayley Graphs}
	\label{subsec:semidirect_products}
	
	We now pass from individual examples to Cayley graphs whose layer structure is induced by a homomorphism onto $\mathbb Z$, and study the coupling condition through the structure of the base group.
	
	Let $G$ be a group equipped with a finite symmetric generating set $S=S^{-1}$, with $e\notin S$, and suppose that there is a surjective homomorphism $\psi:G\to\mathbb{Z}$ such that
	\[
		\psi(S)\subset\{-1,0,1\}.
	\]
	The Cayley graph $\Gamma(G,S)$ has vertex set $G$ and unordered edges $\{g,gs\}$ for $g\in G$ and $s\in S$.
	
	The fibers $V_n=\psi^{-1}(n)$ form a $\mathbb Z$-layered decomposition. With $H=\ker\psi$, surjectivity gives the short exact sequence
	\[
		1 \to H \to G \xrightarrow{\psi} \mathbb{Z} \to 1
	\]
	and identifies $H$ with the base layer $V_0$.
	
	Because $\mathbb{Z}$ is free, the exact sequence splits. Choose $t\in G$ with $\psi(t)=1$; then $n\mapsto t^n$ defines a splitting. Conjugation by $t$ induces $\theta\in\operatorname{Aut}(H)$ through $\theta(h)=t^{-1}ht$, so that $ht=t\theta(h)$. Every $g\in G$ has a unique factorization $g=t^nh$, and in the corresponding coordinates $G\cong\mathbb Z\ltimes_\theta H$ with
	\[
		(n_1,h_1)(n_2,h_2)
		=(n_1+n_2,\theta^{n_2}(h_1)h_2).
	\]
	In particular,
	\[
		(0,h)(1,h_s)=(1,\theta(h)h_s).
	\]
	This convention is adapted to right-multiplicative adjacency. Left translations preserve the right Cayley graph and identify the coordinate matrices of the interlayer coupling operators used below.
	
	\begin{example}
		For the Heisenberg group above, $\psi(x,y,z)=y$, $t=Y$, and $H=\ker\psi=\langle X,Z\rangle\cong\mathbb Z^2$. The induced automorphism is $\theta(x,z)=(x,z+x)$, which is the semidirect-product form of the shear in Figure~\ref{fig:heisenberg_shear}.
	\end{example}
	
	Let $S_+ = S \cap \psi^{-1}(1)$ and $S_-=S\cap\psi^{-1}(-1)$. Since $S$ generates $G$, both are nonempty finite sets. If $S_+=\{s_+\}$, we may choose the splitting element $t=s_+$. Then right multiplication by $s_+$ is a bijection from $V_n$ onto $V_{n+1}$, and its inverse is right multiplication by $s_+^{-1}\in S_-$. Hence the unique-neighbor criterion in Theorem~\ref{thm:mainFinite} holds in both directions. When $|S_+|\geq2$, the determinant is governed by the structure of the base group $H$.
		
	\subsubsection{Finite Base Groups}
	Let $|H|=N<\infty$. Each layer $V_n$ has cardinality $N$, and right multiplication by $S_+$ determines the $N\times N$ coupling matrix $M_{n,n+1}$. Fix an ordering of $H$ and transport it to $V_n$ through $h\mapsto t^nh$. Left translation by $t^{-n}$ then identifies $M_{n,n+1}$ with $M_{0,1}$.
	
	\begin{theorem}
		\label{thm:finite_base}
		Let $G \cong \mathbb{Z} \ltimes_\theta H$, where the base group $H$ is finite of order $N$, and let $S=S^{-1}$ be a finite generating set satisfying $\psi(S)\subset\{-1,0,1\}$. Let $S_+ = S \cap \psi^{-1}(1)$ denote the set of forward generators, and let $M_{0,1}$ be the $N \times N$ coupling matrix between $V_0$ and $V_1$ induced by right multiplication by $S_+$. Then the Cayley graph $\Gamma(G, S)$ satisfies the global coupling condition if and only if $\det(M_{0,1}) \neq 0$.
	\end{theorem}
	
	\begin{proof}
		Every layer has $N$ vertices, so finite-row independence is equivalent to invertibility of the corresponding $N\times N$ matrix. The left-translation identification above shows that every $M_{n,n+1}$ is obtained from $M_{0,1}$ by row and column permutations. Since the graph is undirected,
		\[
			M_{n,n-1}=M_{n-1,n}^{T}.
		\]
		Consequently, $\det M_{0,1}\neq0$ makes all forward and backward coupling matrices invertible. Conversely, the global coupling condition gives finite-row independence of the operator $M_{0,1}$, so its square coupling matrix has full rank and hence nonzero determinant.
	\end{proof}
	
	For example, let $G=\mathbb Z\times\mathbb Z_N$ with $N\geq2$ and forward generators $S_+=\{(1,0),(1,1)\}$. These generators are distinct, and the base coupling matrix is $M_{0,1}=I_N+P_N$, where $P_N$ is the permutation matrix of the $N$-cycle induced by $(1,1)$. Hence
	\[
		\det(M_{0,1})=1+(-1)^{N-1}.
	\]
	Thus the global coupling condition holds when $N$ is odd and fails when $N$ is even. When $N=1$, the two displayed generators coincide in the simple Cayley graph and $M_{0,1}=[1]$, so this degenerate case is excluded from the formula above.
	
	\begin{example}
		For the equation $\Delta u + 4u = f$ on $\mathbb{Z} \times \mathbb{Z}_4$ with $S=\{(1,0),(-1,0),(1,1),(-1,-1)\}$, where the second coordinate is taken modulo $4$, one has
		\[
			(\Delta+4)u(n,k)
			=u(n+1,k)+u(n-1,k)+u(n+1,k+1)+u(n-1,k-1).
		\]
		Summing this identity first over $k\in\{0,2\}$ and then over $k\in\{1,3\}$ yields
		\[
			f(0,0) + f(0,2) = \sum_{k=0}^3 \bigl(u(1,k) + u(-1,k)\bigr) = f(0,1) + f(0,3).
		\]
		Thus the equation has no global solution when $f(0,0)+f(0,2)\neq f(0,1)+f(0,3)$. Figure~\ref{fig:cayley_Z_Z4} displays the alternating interlayer pattern behind this compatibility condition. Failure of the global coupling condition may therefore obstruct even unrestricted linear solvability for suitable source terms. It does not imply nonsolvability for arbitrary nonlinear equations on the graph.
	\end{example}

		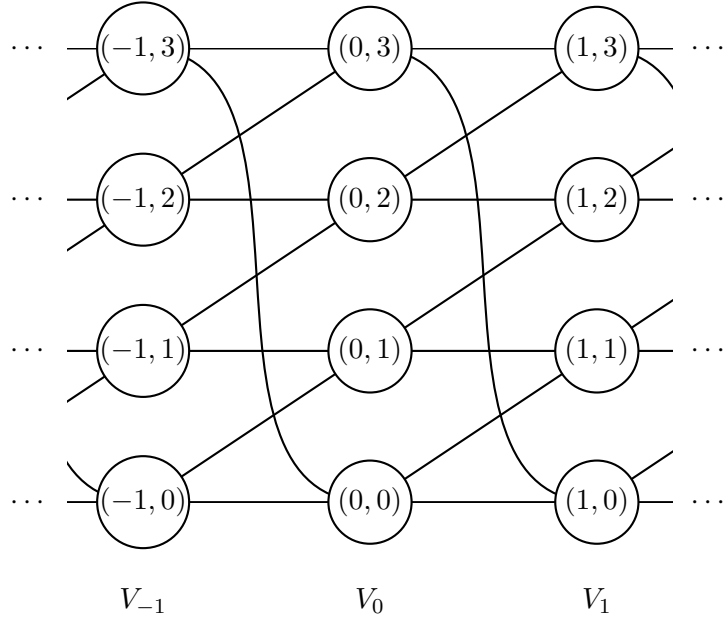
\begin{figure}[htbp]
			\centering
			\begin{tikzpicture}[thick, vnode/.style={circle, draw, fill=white, minimum size=1.1cm, inner sep=0pt}]
				\foreach \n in {-2, -1, 0, 1, 2} {
					\foreach \h in {0, 1, 2, 3} {
						\coordinate (v\n\h) at (3*\n + 3, 2*\h);
					}
				}
				\begin{scope}
					\clip (-1, 0) rectangle (7, 6);
					\foreach \n in {-2, -1, 0, 1} {
						\pgfmathtruncatemacro{\nextn}{\n + 1}
						\foreach \h in {0, 1, 2, 3} {
							\draw (v\n\h) -- (v\nextn\h);
							\pgfmathtruncatemacro{\nexth}{mod(\h + 1, 4)}
							\ifnum\h=3
							\draw (v\n\h) to[out=0, in=180] (v\nextn\nexth);
							\else
							\draw (v\n\h) -- (v\nextn\nexth);
							\fi
						}
					}
				\end{scope}
				\foreach \n in {-1, 0, 1} {
					\foreach \h in {0, 1, 2, 3} {
						\node[vnode] at (v\n\h) {$(\n, {\h})$};
					}
					\node[below] at (3*\n + 3, -1) {$V_{\n}$};
				}
				\foreach \h in {0, 1, 2, 3} {
					\node[fill=white, inner sep=2pt] at (-1.5, 2*\h) {$\cdots$};
					\node[fill=white, inner sep=2pt] at (7.5, 2*\h) {$\cdots$};
				}
			\end{tikzpicture}
			\caption{Local structure of the Cayley graph of $\mathbb Z\times\mathbb Z_4$. The two forward neighbors of each vertex make the singular interlayer coupling matrix and the alternating compatibility obstruction visible.}
			\label{fig:cayley_Z_Z4}
		\end{figure}
	\FloatBarrier
	
	\subsubsection{Left-Orderable Base Groups}
	Assume the base group $H$ is left-orderable, admitting a strict total order $\prec$ invariant under left translation: $a \prec b$ implies $c a \prec c b$ for all $a,b,c\in H$. We write $\preceq$ for the associated non-strict order.
	
	\begin{theorem} \label{thm:left_orderable_surjectivity}
		Let $G \cong \mathbb{Z} \ltimes_\theta H$, and let $S=S^{-1}$ be a finite generating set satisfying $\psi(S)\subset\{-1,0,1\}$. If $H$ is left-orderable, then the Cayley graph $\Gamma(G, S)$ satisfies the global coupling condition. Consequently, for every function $f:G\times\mathbb R\to\mathbb R$, the equation $\Delta u(x) = f(x,u(x))$ admits infinitely many solutions.
	\end{theorem}
	
	\begin{proof}
		Write $s=(1,h_s)$ for $s\in S_+$, and set
		\[
			h_*:=\max\{h_s:s\in S_+\},
			\qquad s_*:=(1,h_*).
		\]
		
		For $v=(0,h)\in V_0$, right multiplication by $s\in S_+$ gives
		\[
			vs=(0,h)(1,h_s)=(1,\theta(h)h_s).
		\]
		The left-invariance of $\prec$ implies $\theta(h) h_s \prec \theta(h) h_*$ for all $s \in S_+ \setminus \{s_*\}$. 
		
		Let $U=\{u_1,\ldots,u_m\}\subset V_0$ be finite, write $u_i=(0,h_i)$, and put $K=Us_*$. Right multiplication by $s_*$ is bijective, so $|K|=|U|$. Define
		\[
			z_i:=\theta(h_i)h_*,
		\]
		the $H$-coordinate of $u_is_*$. Index the pairs $(u_i,u_is_*)$ so that $z_1\prec\cdots\prec z_m$. If the entry in row $i$ and column $j$ is nonzero, then $u_is=u_js_*$ for some $s=(1,h_s)\in S_+$, and therefore
		\[
			z_j=\theta(h_i)h_s\preceq\theta(h_i)h_*=z_i.
		\]
		Hence $j\leq i$, and the restricted adjacency matrix is lower triangular with unit diagonal, supplied by $s_*$. Its determinant is $1$.
		
		By Theorem~\ref{thm:mainFinite}, $M_{0,1}$ satisfies the finite-row independence condition. Write every $s\in S_-$ as $s=(-1,h_s)$ and let $h_-^*$ be the maximum of the finite set of its $H$-coordinates. Since
		\[
			(0,h)(-1,h_s)=(-1,\theta^{-1}(h)h_s),
		\]
		the same triangular argument, with $\theta$ replaced by $\theta^{-1}$, proves finite-row independence of $M_{0,-1}$. Left translations by powers of $t$ carry these conclusions to every layer. The global coupling condition follows, and Theorem~\ref{thm:mainSur} gives infinitely many solutions.
	\end{proof}
	
	\begin{example}
		For $G = \mathbb{Z}^d \cong \mathbb{Z} \times \mathbb{Z}^{d-1}$, the base group $\mathbb Z^{d-1}$ is left-orderable, for instance by the lexicographic order. Theorem~\ref{thm:left_orderable_surjectivity} therefore applies to every finite symmetric generating set $S$ satisfying $\psi(S)\subset\{-1,0,1\}$ for the first-coordinate homomorphism $\psi$, not only to the standard orthogonal and triangular choices considered above.
	\end{example}
	
	\subsubsection{Split Mixed Base Groups}
	Assume that $H$ fits into a split exact sequence
	\[
		1\longrightarrow F\longrightarrow H\xrightarrow{q}T\longrightarrow1,
	\]
	where $F$ is finite of order $N$ and $T$ is left-orderable with order $\prec$. Fix a splitting $\sigma:T\to H$ and write each $h\in H$ uniquely as $h=\sigma(x)y$, with $(x,y)\in T\times F$. Thus $H\cong T\ltimes F$, with the right-hand factor normal; explicitly,
	\[
		(\sigma(x)y)(\sigma(x')y')
		=\sigma(xx')\bigl(\sigma(x')^{-1}y\sigma(x')\bigr)y'.
	\]
	
	Since $T$ is left-orderable, it is torsion-free, so $F$ is exactly the set of torsion elements of $H$. Automorphisms preserve element orders; hence $F$ is characteristic in $H$, and therefore $\theta(F)=F$. Thus $\theta$ induces $\bar\theta\in\operatorname{Aut}(T)$ satisfying $q\circ\theta=\bar\theta\circ q$. For $x\in T$ and $n\in\mathbb Z$, write
	\[
		\mathcal F_x^{(n)}:=\{(n,\sigma(x)y):y\in F\}\subset V_n
	\]
	for the corresponding $F$-fiber in the $n$th layer.
	
	For the forward direction, write each $s\in S_+=S\cap\psi^{-1}(1)$ uniquely as $s=(1,h_s)$, where $h_s=\sigma(x_s)y_s$, and let
	\[
		x_*:=\max\{x_s:s\in S_+\},
		\qquad S_*:=\{s\in S_+:x_s=x_*\}.
	\]
	Every $s\in S_*$ maps $\mathcal F_e^{(0)}$ bijectively onto $\mathcal F_{x_*}^{(1)}$ by right multiplication. With arbitrary orderings of these fibers, let $D_F$ be the resulting $N\times N$ adjacency matrix generated by $S_*$, called the \emph{extremal $F$-fiber block}. Its nonsingularity is independent of the chosen orderings.
	
	\begin{theorem} \label{thm:mixed_base_surjectivity}
		Let $G\cong\mathbb Z\ltimes_\theta H$ and let $S=S^{-1}$ be a finite generating set satisfying $\psi(S)\subset\{-1,0,1\}$. Suppose that $H$ has the split form $H\cong T\ltimes F$ described above, where $F$ is finite and $T$ is left-orderable. If the extremal $F$-fiber block is nonsingular,
		\[
			\det D_F\neq0,
		\]
		then $\Gamma(G,S)$ satisfies the global coupling condition.
	\end{theorem}
	
	\begin{proof}
		Choose $s_*\in S_*$. For $v=(0,\sigma(x)y)\in\mathcal F_x^{(0)}$, right multiplication by $s\in S_+$ gives
		\[
			vs=(1,\theta(\sigma(x)y)h_s)\in\mathcal F_{\bar\theta(x)x_s}^{(1)}.
		\]
		The left-invariance of $\prec$ implies $\bar\theta(x)x_s\prec\bar\theta(x)x_*$ for all $s\in S_+\setminus S_*$.
		
		Let $U_0\subset V_0$ be finite, enlarge it to the union
		\[
			U=\bigsqcup_{i=1}^m\mathcal F_{x_i}^{(0)}
		\]
		of the $F$-fibers that it meets, and put $K=Us_*$. Right multiplication by $s_*$ is bijective, so $|K|=|U|=mN$. Define
		\[
			z_i:=\bar\theta(x_i)x_*,
			\qquad
			\mathcal F_{x_i}^{(0)}s_*=\mathcal F_{z_i}^{(1)}.
		\]
		The $z_i$ are distinct because $x\mapsto\bar\theta(x)x_*$ is bijective. Index the fiber pairs $(\mathcal F_{x_i}^{(0)},\mathcal F_{z_i}^{(1)})$ so that $z_1\prec\cdots\prec z_m$. If the $(i,j)$-block of $M_{U,K}$ is nonzero, then for some $s\in S_+$,
		\[
			z_j=\bar\theta(x_i)x_s\preceq\bar\theta(x_i)x_*=z_i.
		\]
		Hence $j\leq i$, and $M_{U,K}$ is block lower triangular.
		
		The diagonal blocks are generated precisely by $S_*$. Left translation by $(0,\sigma(x_i))$ maps the pair $(\mathcal F_e^{(0)},\mathcal F_{x_*}^{(1)})$ to $(\mathcal F_{x_i}^{(0)},\mathcal F_{z_i}^{(1)})$ and preserves the right-generating labels. Thus each diagonal block equals $D_F$ up to row and column permutations, and
		\[
			\det M_{U,K}=\pm(\det D_F)^m\neq0.
		\]
		The rows indexed by $U$, and hence those indexed by $U_0$, are therefore independent in $M_{0,1}$. This proves finite-row independence of $M_{0,1}$.
		
		For the opposite direction, let $W_0\subset V_1$ be finite, enlarge it to the union $W$ of the $F$-fibers that it meets, and put $X=Ws_*^{-1}$. Since
		\[
			\mathcal F_z^{(1)}s_*^{-1}
			=\mathcal F_{\bar\theta^{-1}(zx_*^{-1})}^{(0)},
		\]
		$X$ is a finite union of $F$-fibers and $W=Xs_*$. The preceding construction, applied to $U=X$ and $K=W$, gives $\det M_{X,W}\neq0$. Thus the columns indexed by $W_0$ are independent in $M_{0,1}$. Since $W_0\subset V_1$ was arbitrary, $M_{1,0}=M_{0,1}^{T}$ has finite-row independence. Left translations by powers of $t$ carry these conclusions to every layer. The global coupling condition follows.
	\end{proof}
	\FloatBarrier
	\section{Further Extensions}
	
	The layer-by-layer extension method also applies to several related operators. Given $a>0$ and a potential $\rho:V\to\mathbb R$, define $L_{a,\rho}$ by
	\[
		(L_{a,\rho}u)(x):=-a\Delta u(x)+\rho(x)u(x).
	\]
	For every $F:V\times\mathbb R\to\mathbb R$, the equation $(L_{a,\rho}u)(x)=F(x,u(x))$ is equivalent to
	\[
		\Delta u(x)=\frac{\rho(x)u(x)-F(x,u(x))}{a},
	\]
	so Theorem~\ref{thm:mainSur} applies under the same global coupling condition. The argument extends further to row-finite operators of hopping range one whose requisite interlayer coupling operators have finite-row independence relative to the chosen decomposition. Positive weighted graph Laplacians are governed by weighted coupling operators; unique matchings in finite submatrices of their coupling matrices remain sufficient, whereas the Cayley-graph determinant classifications require translation-compatible weights.

	The remaining variants require inversion of different objects: an outer two-step coupling operator for $\Delta^2$, the scalar bijection $\Phi_p$ under a unique-neighbor hypothesis for $\Delta_p$, and complex interlayer coupling operators satisfying the same finite-row criterion for magnetic Laplacians.

	\subsection{Bi-Laplacian}
	Write $A_n:=\Delta_n$ for the diagonal block of $\Delta$ on $V_n$. Squaring the block-tridiagonal representation of $\Delta$ gives
	\begin{align}
		(\Delta^2u)|_{V_n}
		={}&M_{n,n-1}M_{n-1,n-2}u_{n-2} \nonumber\\
		&+\bigl(M_{n,n-1}A_{n-1}+A_nM_{n,n-1}\bigr)u_{n-1} \nonumber\\
		&+\bigl(M_{n,n-1}M_{n-1,n}+A_n^2+M_{n,n+1}M_{n+1,n}\bigr)u_n \label{eq:bilaplacian_blocks}\\
		&+\bigl(A_nM_{n,n+1}+M_{n,n+1}A_{n+1}\bigr)u_{n+1} \nonumber\\
		&+M_{n,n+1}M_{n+1,n+2}u_{n+2}, \nonumber
	\end{align}
	with terms involving nonexistent layers omitted at the boundary of an $\mathbb N$-layered decomposition. The outer two-step coupling operators are
	\[
		B_{n,n+2}^+:=M_{n,n+1}M_{n+1,n+2},
		\qquad
		B_{n,n-2}^-:=M_{n,n-1}M_{n-1,n-2}.
	\]
	The diagonal blocks in~\eqref{eq:bilaplacian_blocks} do not contribute to the outermost layers.
	
	\begin{proposition}
		\label{prop:bilaplacian_extension}
		Let $G$ be a connected, locally finite graph with an $\mathcal I$-layered decomposition, and let $f:V\times\mathbb R\to\mathbb R$.

		If $\mathcal I=\mathbb N$ and $B_{n,n+2}^+$ satisfies the finite-row independence condition for every $n\geq0$, then arbitrary data on $V_0$ and $V_1$ extend to a global solution of $\Delta^2u(x)=f(x,u(x))$.
		
		If $\mathcal I=\mathbb Z$ and both $B_{n,n+2}^+$ and $B_{n,n-2}^-$ satisfy the finite-row independence condition for every $n\in\mathbb Z$, then for every $k\in\mathbb Z$, arbitrary data on the four consecutive layers $V_k,V_{k+1},V_{k+2},V_{k+3}$ extend to a global solution. In either case the equation has infinitely many solutions.
	\end{proposition}
	
	\begin{proof}
		A connected locally finite graph is countable, being the union of the finite balls about a fixed root. The argument of Proposition~\ref{prop:eig_solvability} therefore applies to any continuous row-finite linear map $A:\mathbb R^X\to\mathbb R^Y$ between the layerwise product spaces. Local finiteness makes each two-step operator $B_{n,n+2}^+$ and $B_{n,n-2}^-$ continuous and row-finite; the assumed finite-row independence makes these operators surjective. In the $\mathbb N$-layered case, prescribe arbitrary $(u_0,u_1)$. Once $u_0,\ldots,u_{n+1}$ have been chosen, equation~\eqref{eq:bilaplacian_blocks} on $V_n$ has the form
		\[
			B_{n,n+2}^+u_{n+2}=f(\cdot,u_n)-Q_n(u_{n-2},u_{n-1},u_n,u_{n+1}),
		\]
		where $Q_n$ is the sum of the already known terms, with nonexistent negative-index layers omitted. Surjectivity of $B_{n,n+2}^+$ permits a choice of $u_{n+2}$, and induction enforces the equation on every layer.
		
		For the $\mathbb Z$-layered case, prescribe arbitrary data on $V_k,V_{k+1},V_{k+2},V_{k+3}$. Surjectivity of $B_{k+1,k-1}^-$ provides $u_{k-1}$ solving the equation on $V_{k+1}$, while surjectivity of $B_{k+2,k+4}^+$ provides $u_{k+4}$ solving the equation on $V_{k+2}$. Proceeding through the equations on $V_k,V_{k-1},\ldots$ to the left and on $V_{k+3},V_{k+4},\ldots$ to the right successively selects further outer layers. At each step the new outer layer enters no equation already enforced. This produces a global solution for every prescribed four-layer datum. The preimages need not be unique, and distinct initial data yield distinct solutions.
	\end{proof}
	
	For the decomposition $\mathbb Z^d\cong\mathbb Z\times\mathbb Z^{d-1}$, all adjacent-layer and two-step coupling operators are identities under the natural identifications; Proposition~\ref{prop:bilaplacian_extension} applies.
	
	\Needspace{22\baselineskip}
	\subsection{\texorpdfstring{$p$}{p}-Laplacian}
	For $p>1$, let $\Phi_p(t):=|t|^{p-2}t$, with $\Phi_p(0):=0$, and define the discrete $p$-Laplacian by
	\[
		\Delta_p u(x):=\sum_{y\sim x}\Phi_p\bigl(u(y)-u(x)\bigr).
	\]
	We use the standard convention for the graph $p$-Laplacian~\cite{HolopainenSoardi1997}. The dependence on the next-layer variables is nonlinear when $p\neq2$. Under a unique-neighbor hypothesis, however, each equation contains exactly one new value and can be inverted.
	
	\begin{theorem}
		\label{thm:p_laplacian}
		Let $p>1$, and let $G=(V,E)$ be a connected, locally finite, countably infinite graph with an $\mathcal I$-layered decomposition. For every $n\geq0$, assume that there is an injection $\phi_n^+:V_n\to V_{n+1}$ satisfying
		\[
			N(\{x\})\cap V_{n+1}=\{\phi_n^+(x)\},\qquad x\in V_n.
		\]
		In the $\mathbb Z$-layered case, assume likewise that for every $n\leq-1$ there is an injection $\phi_n^-:V_n\to V_{n-1}$ satisfying
		\[
			N(\{x\})\cap V_{n-1}=\{\phi_n^-(x)\},\qquad x\in V_n.
		\]

		Then, for every $f:V\times\mathbb R\to\mathbb R$, arbitrary data on $V_0$ in the $\mathbb N$-layered case, or on $V_{-1}\cup V_0$ in the $\mathbb Z$-layered case, extend to a global solution of $\Delta_pu(x)=f(x,u(x))$. In particular, the equation has infinitely many solutions.
	\end{theorem}
	
	\begin{proof}
		The map $\Phi_p$ is a strictly increasing bijection of $\mathbb R$ with
		\[
			\Phi_p^{-1}(s)=\operatorname{sgn}(s)|s|^{1/(p-1)}.
		\]
		For $\mathcal I=\mathbb N$, choose $u_0$ arbitrarily. Once $u_{n-1}$ and $u_n$ are known, with $V_{-1}=\varnothing$ when $n=0$, the equation at $x\in V_n$ uniquely determines
		\[
			u_{n+1}(\phi_n^+(x))
			=
			u_n(x)+\Phi_p^{-1}\!\left(
			f(x,u_n(x))-
			\sum_{\substack{y\sim x\\y\in V_{n-1}\cup V_n}}
			\Phi_p\bigl(u(y)-u_n(x)\bigr)
			\right).
		\]
		Injectivity makes these assignments consistent. The singleton-neighbor condition also implies that vertices outside $\operatorname{im}\phi_n^+$ have no neighbors in $V_n$, so the remaining values on $V_{n+1}$ are arbitrary and do not alter any equation already enforced. Iteration gives a global solution.
		
		For $\mathcal I=\mathbb Z$, choose $(u_{-1},u_0)$ arbitrarily and use the same formula with $\phi_n^+$ to extend forward. The backward extension follows analogously using $\phi_n^-$ for $n\leq-1$. Varying the prescribed pair yields infinitely many global solutions.
	\end{proof}
	
	When next-layer variables are shared by several equations, scalar inversion need not decouple the system. For $p=2$, linearity gives the extension conclusion under finite-row independence by Theorem~\ref{thm:mainSur}; for $p\neq2$, the linear finite-row argument does not directly provide the required nonlinear inversion. The unique-neighbor hypothesis of Theorem~\ref{thm:p_laplacian} holds for $\mathbb Z^d\cong\mathbb Z\times\mathbb Z^{d-1}$, the brick-wall hexagonal graph, and the layered Cayley graph of the discrete Heisenberg group.
	
	\Needspace{10\baselineskip}
	\subsection{Magnetic Laplacians}
	We use the standard convention of~\cite{DodziukMathai2006}. With our sign convention, let
	\[
		\vec E:=\{(x,y)\in V\times V:x\sim y\}
	\]
	be the set of oriented edges. A magnetic phase is a map
	\[
		\alpha:\vec E\to\mathbb S^1
		\qquad\text{such that}\qquad
		\alpha_{yx}=\overline{\alpha_{xy}}=\alpha_{xy}^{-1},
		\qquad |\alpha_{xy}|=1,
	\]
	where $\mathbb S^1=\{z\in\mathbb C:|z|=1\}$. For $u\in\mathbb C^V$, define the magnetic Laplacian by
	\[
		\Delta_\alpha u(x)
		:=
		\sum_{y\sim x}\bigl(\alpha_{xy}u(y)-u(x)\bigr).
	\]
	The layered argument is unchanged: the coupling operators act on complex-valued functions, with matrix entries
	\[
		M_{i,j}^\alpha(x,y)
		=
		\begin{cases}
			\alpha_{xy},&x\sim y,\\
			0,&x\not\sim y.
		\end{cases}
	\]
	By the complex form of Lemma~\ref{lem:eidelheit_strict}, $M_{i,j}^\alpha$ is surjective if and only if every finite row restriction of its coupling matrix has full complex row rank. Equivalently, for every finite $U\subset V_i$ there is $K\subset V_j$, with $|K|=|U|$, such that $\det M_{U,K}^\alpha\neq0$. This is the interlayer analogue of the magnetic Schr\"odinger criterion in~\cite{KobersteinSchmidt2020}.
	
	The phase-weighted determinant terms of distinct perfect matchings may cancel. For a unique matching $\pi:U\to K$, however, one has
	\[
		\det M_{U,K}^\alpha
		=
		\operatorname{sgn}(\pi)
		\prod_{x\in U}\alpha_{x,\pi(x)},
		\qquad
		\bigl|\det M_{U,K}^\alpha\bigr|=1.
	\]
	Every preceding application based on a unique matching therefore survives an arbitrary magnetic phase. If all required magnetic coupling operators have finite-row independence over $\mathbb C$, the proof of Theorem~\ref{thm:mainSur} extends arbitrary complex data on $V_0$ in the one-sided case, or on $V_{-1}$ and $V_0$ in the two-sided case, to a solution of $\Delta_\alpha u(x)=F(x,u(x))$ for every $F:V\times\mathbb C\to\mathbb C$. This conclusion need not follow from an unphased nonsingular minor when several perfect matchings contribute, since their phase-weighted determinant terms may cancel.
	
	The conclusion also covers $H_{\alpha,q}:=-\Delta_\alpha+q$ for $q:V\to\mathbb R$, since $H_{\alpha,q}u=F(x,u(x))$ is equivalent to $\Delta_\alpha u=q(x)u-F(x,u(x))$.
	
	All extension results above are formulated in the product topology and provide no a priori control of positivity, growth, summability, or energy. These results leave open whether the coupling operators admit quantitative right inverses in weighted $\ell^p$ or prescribed-growth spaces and which genuinely nonlinear interlayer systems remain solvable beyond the unique-neighbor setting.
	
	\section*{Acknowledgments}
	We thank Dong Ye for stimulating discussions on nonlinear elliptic equations and for motivating the solvability question for $\Delta u = e^u$ on lattice graphs. B.H. is supported by the National Natural Science Foundation of China (Grant No. 12371056). Y.L. is supported by the National Natural Science Foundation of China (Grant Nos. 12071245 and 12471088).
	
	\bibliographystyle{plain}
	\bibliography{EIG}
	
\end{document}